\documentclass[a4paper,12pt]{article}
\usepackage{babel}
\usepackage[utf8]{inputenc}
\usepackage{amsthm}
\usepackage{amsfonts}
\usepackage{amssymb}
\usepackage{amsmath}
\usepackage{hyperref}
\usepackage{graphicx}
\usepackage{rotating}
\usepackage{wrapfig}
\usepackage[all]{xy}
\usepackage{pgf,tikz}
\usetikzlibrary{arrows}
\usetikzlibrary[patterns]

\newcommand{\N}{\mathbb{N}}

\newtheorem{prop}{Proposition}[section]

\newtheorem{obs}{Remark}[section]

\newtheorem{defi}[prop]{Definition}
\newtheorem{cor}[prop]{Corollary}
\newtheorem{teo}[prop]{Theorem}
\newtheorem{lema}[prop]{Lemma}
\newtheorem{ex}[prop]{Example}

\newtheoremstyle{bloco}
  {\topsep}   
  {\topsep}   
  {\upshape}  
  {0pt}       
  {\bfseries} 
  {\!\!\!\!}         
  {5pt plus 1pt minus 1pt} 
\makeatother

\theoremstyle{bloco}
\newtheorem{bloco}{}

\title{Automatic continuity for vector spaces with linear topology}
\author{Samuel Quirino and Lucas H. R. de Souza}

\begin{document}

\DeclareGraphicsExtensions{.pdf,.jpg,.mps,.png,}

\maketitle

\begin{abstract}In this paper we classify all topological vector spaces with linear topology with the property that all algebraic automorphisms are continuous. Moreover, we prove some properties of these spaces.
\end{abstract}


\



\

\textbf{Data availability statement} \ There is no data related to this paper.

\


\section*{Introduction}

In \cite{So9} there is a classification of all topological spaces $X$ with the property that every bijection from $X$ to itself is continuous. These spaces are precisely the spaces with the co-$\kappa$ topology, i.e. fixed the cardinal $\kappa$, the closed sets of the space are the sets with cardinality less than $\kappa$. Similarly, in Functional Analysis there is the concept of automatic continuity, i. e. when a given algebraic homomorphism between  Banach algebras is already continuous (see \cite{Dal} for the basics of this theory). For example, there is this strong result, due to Johnson, which says that if $f: A \rightarrow B$ is a homomorphism of Banach algebras and $B$ is semisimple, then $f$ is continuous (Theorem 5.1.5 of \cite{Dal}).

In this paper we consider topological vector spaces with linear topology, i.e. a topology where the neighborhood filter of $0$ is generated by open vector subspaces. In the same spirit as the one given above, we classify all vector spaces with linear topology with the property that all algebraic automorphisms are continuous.

Let $V$ be a vector space and $\kappa$ be a cardinal with $\kappa = 1$ or $\kappa \geqslant \aleph_{0}$. The $\kappa$-codimension topology of $V$ is the topology induced by the filter generated by all vector subspaces with codimension less than $\kappa$. With this topology, $V$ is a topological vector space and the topology, by construction, is linear. We say that $V$ has codimension topology if it has the $\kappa$-codimension topology for some cardinal $\kappa$. Let $Aut(V)$ be the set of algebraic automorphisms of $V$ (i.e. isomorphisms as vector spaces) and $CAut(V)$ is its subset given by continuous maps. Then we have our main result:

\begin{bloco}{\textbf{Theorem \ref{theoremone}}} \textit{Let $V$ be a topological vector space with linear topology. Then $V$ has a codimension topology if and only if $Aut(V) = CAut(V)$.  }
\end{bloco}

This theorem is analogous to Theorem 1 of \cite{So9}. If $V$ and $V'$ are topological vector spaces, then we define $Hom(V,V')$ as the set of linear maps from $V$ to $V'$ and $CHom(V,V')$ the set of continuous linear maps from $V$ to $V'$. We proved also the following two propositions:

\begin{bloco}{\textbf{Proposition \ref{continuitytosmallspaces}}} \textit{Let $V$ and $V'$ be topological vector spaces with linear topology such that $V$ has the $\kappa$-codimension topology. If $dim \ V' < \kappa$, then
$CHom(V,V') = Hom(V,V')$.}    
\end{bloco}

\begin{bloco}{\textbf{Proposition \ref{continuitybetweenboundedspaces}}} \textit{Let $V$ and $V'$ be topological vector spaces such that $V$ and $V'$ have the $\kappa$ and $\kappa'$-codimension topologies, respectively. If $\kappa' \leqslant \kappa$, then $CHom(V,V') = Hom(V,V')$.} 
\end{bloco}

Let $V = V'$ and both spaces endowed with the $\kappa$-codimension topology. Then, by \textbf{Proposition \ref{continuitybetweenboundedspaces}} $CHom(V,V) = Hom(V,V)$. This phenomenon is different from the one in topological spaces, where the spaces with the property that all maps from $X$ to itself are continuous maps are the discrete and the trivial spaces (Corollary 1 of \cite{So9}). 

Also, we show a series of properties of these topological vector spaces. With this, we construct an arbitrarily big family of topological vector spaces with linear topology which are all non isomorphic as topological vector spaces and with the property that all of them have the same profinite-dimensional completion and the same dual space (\textbf{Corollary \ref{bigfamily}}).

At the end, we consider analogous versions of Johnson's Theorem for topological algebras with linear topology. We get that, in general, such statement fails (\textbf{Example \ref{example1}}) and if we restrict it to pseudocompact algebras, then the statement also fails (\textbf{Example \ref{example2}}). 

\section*{Acknowledgements}

We would like to thank Geraldo Botelho for the discussion about Automatic Continuity, which was the starting point for the idea of this paper. We also thank John MacQuarrie for the discussion about the paper. In particular, he showed us \textbf{Proposition \ref{linearcompacthasdense}}, which helped a lot with the comprehension of the spaces that appear here.

\section{Preliminaries}

\subsection{Linear algebra and cardinals}

If $V$ is a $k$-vector space and $W < V$, then the codimension of $W$ in $V$ is given by $dim \ V/W$ and denoted as $codim_{V}W$. If the ambient space is clear from its context, then we denote it simply as $codim \ W$.

\begin{prop}\label{codimensionofintersection}Let $V$ be a vector space, $W,W' < V$ and $\kappa$ a cardinal number satisfying $\kappa = 1$ or $\kappa \geqslant \aleph_{0}$. If $codim \ W < \kappa$ and $codim \ W' < \kappa$, then $codim \ W \cap W' < \kappa$.
\end{prop}

\begin{obs}This result certainly is well known. At least for the case $\kappa = \aleph_{0}$, it  appears as Exercise 13 of page 91 of \cite{Rom}. However, the standard proof does not depend on $\kappa$, if $\kappa \geqslant \aleph_{0}$. We write its proof down here, for the sake of completeness.
\end{obs}

\begin{proof}For $\kappa = 1$ the result is immediate. Suppose that $\kappa \geqslant \aleph_{0}$. We have that $dim \ (W+W')/W' \leqslant dim \ V/W' = codim_{V}W' < \kappa$. One of the Isomorphism Theorems says that $$(W+W')/W' \cong W/(W\cap W').$$ So $dim \ W/(W\cap W') < \kappa$. Another Isomorphism Theorem gives us that $$V/W \cong (V/W\cap W')/(W/W\cap W'),$$ which implies that $V/W\cap W' \cong (V/W) \oplus (W/W\cap W')$ and then $$dim \ (V/W\cap W') = dim \ (V/W) + dim \ (W/W\cap W').$$ So $codim_{V} W\cap W' = dim \ (V/W\cap W')  = dim \ (V/W) + dim \ (W/W\cap W') < \kappa + \kappa = \kappa$.
\end{proof}

\begin{prop}\label{preimagesandcodimension}Let $f: V \rightarrow V'$ be a linear map, $\kappa$ a cardinal and $W' < V$ such that $codim \ W' < \kappa$. Then $codim \ f^{-1}(W) < \kappa$. 
\end{prop}

\begin{proof}Take $\pi: V' \rightarrow V'/W'$ as the quotient map. Then $\ker \ \pi \circ f = f^{-1}(W')$. Then the First Isomorphism Theorem gives us that there is an injective homomorphism $V/f^{-1}(W') \rightarrow V'/W'$, which implies that $codim \ f^{-1}(W') = dim \ V/f^{-1}(W') \leqslant dim \ V'/W' = codim \ W' < \kappa$.
\end{proof}

Here we establish a notation: if $\kappa$ is a cardinal, then $\kappa^{+}$ is its cardinal successor. In addition, any subspace of a vector space means a vector subspace. We say subset if it may not have structure of vector subspace.

\subsection{Vector spaces with linear topology}

On this entire paper, we fix a discrete field $k$.

\begin{prop}(Theorem 12.3 of \cite{Wa})Let $V$ be a vector space and $\mathcal{N}$ a filter on $V$ which satisfies the following properties:

\begin{enumerate}
    \item For every $U \in \mathcal{N}$, $0 \in U$.
    \item For each $U \in \mathcal{N}$, there exists $U' \in \mathcal{N}$ such that $U'+U' \subseteq U$.
    \item For each $U \in \mathcal{N}$, there exists $U' \in \mathcal{N}$ such that $U' \subseteq - U$.
    \item For each $U \in \mathcal{N}$ and $a \in k$, there exists $U' \in \mathcal{N}$ such that $aU' \subseteq U$.
\end{enumerate}

Then there is a unique topology on $V$ such that $V$ is a topological vector space and $\mathcal{N}$ is the neighborhood filter of $0$.
\end{prop}

The differences from the above proposition and the one given in the reference is from the fact that, in our case, the base field $k$ is discrete, which simplifies some conditions of the filter.

\begin{defi}Let $V$ be a topological vector space. We say that $V$ has linear topology if the neighborhood filter of $0$ is generated by open vector subspaces of $V$.
\end{defi}

\begin{defi}Let $V$ be a topological vector space with linear topology. We say that $V$ is linearly compact if it is an inverse limit of finite dimensional vector spaces, with the topology given by the inverse limit.
\end{defi}

Proposition 24.4 of \cite{BeH} gives us many equivalent definitions of the definition above. Some of them are necessary for this paper.  

\begin{prop}\label{linearcompacthasdense} Let $V$ be an infinite dimensional linearly compact vector space. Then there exists a subspace $W < V$ such that $W$ is not open in $V$ and $codim \ W = 1$.
\end{prop}

\begin{obs}This proposition is an analogous version of the construction given in page 5 of \cite{IM} for pseudocompact algebras and its proof is essentially the same. We write it down here for the sake of completeness.
\end{obs}

\begin{proof}Proposition 24.4 of \cite{BeH} says that $V$ is isomorphic, as topological vector spaces, to $\prod_{i \in \Gamma} k$, for some infinite $\Gamma$. Then the subspace given by $\bigoplus_{i \in \Gamma} k$ is dense on $\prod_{i \in \Gamma} k$. Take $W$ a subspace of $\prod_{i \in \Gamma} k$ such that $\bigoplus_{i \in \Gamma} k < W$ and $codim \ W = 1$ . Then $W$ is also open in $\prod_{i \in \Gamma} k$. So $W$ cannot be open in $\prod_{i \in \Gamma} k$, since any open subspace must also be closed (Corollary 2.3 of \cite{Wa}).
\end{proof}

\begin{defi}Let $V$ be a topological vector space and $\mathcal{N}(0)$ the neighborhood filter of $0$. We denote by $\hat{V}$, the Cauchy completion of $V$, the space, with the inverse limit topology, given by $\lim\limits_{\longleftarrow} \{V/W: W \in \mathcal{N}(0), W \text{ is an open subspace}\}$.
\end{defi}

In \cite{Wa} there is the equivalence between the definition above and the usual definition using uniform spaces.

\begin{defi}Let $V$ be a topological vector space. Its profinite-dimension completion is the limit $\lim\limits_{\longleftarrow}\{V/W: W \text{ is closed}, codim \ W < \aleph_{0}\}$.
\end{defi}

Finally, the topological dual of a vector space with linear topology is defined as:

\begin{defi}Let $V$ be a vector space with linear topology. We define $V_{\kappa}^{\circledast} = CHom(V_{\kappa},k)$, the set of continuous linear maps, together with the topology given by the neighborhood filter of $0$ that is generated by the set $\{W^{\perp}: W < V, W \text{ is linearly compact}\}$, where $W^{\perp} = \{f \in V_{\kappa}^{\circledast}: f(W) = 0\}$.
\end{defi}

 \section{Codimension topology}

\begin{defi}Let $V$ be a vector space and $\kappa$ a cardinal, with $\kappa = 1$ or $\kappa \geqslant \aleph_{0}$. The $\kappa$-codimension topology of $V$ is the topology whose neighborhood filter of the point $0$ is generated by the set $\{W < V: codim \ W < \kappa\}$. We denote by $\tau_{V,\kappa}$ such topology. We say that $V$ has codimension topology if it has the $\kappa$-codimension topology for some cardinal $\kappa$.
\end{defi}

By \textbf{Proposition \ref{codimensionofintersection}}, the set $\{W < V: codim \ W < \kappa\}$ is closed by intersections, which implies that it is a base of a neighborhood filter of the point $0 \in V$ that generates a topology such that $V$ is a topological vector space with linear topology. Note that $\kappa$ cannot be finite, except if it is equal to $1$, since we do not have control on the codimension of the intersections. For example, if $V$ has a basis given by $\{b_{1},b_{2}\}$, then $\langle b_{1}\rangle$ and $\langle b_{2}\rangle$ have codimension $1$ in $V$, but $0 = \langle b_{1}\rangle \cap \langle b_{2}\rangle$ has codimension $2$.

We have two special cases: the $1$-codimension topology is the trivial topology and the $\kappa$-codimension topology, with $\kappa \geqslant dim \ V$ is the discrete topology.

\begin{prop}\label{differentspaces}Let $V$ be a vector space and $\kappa,\kappa'$ cardinals such that $\kappa \leqslant \kappa' \leqslant dim \ V$. Then $\tau_{V,\kappa}$ is coarser than $\tau_{V,\kappa'}$. Moreover, if $\kappa \neq \kappa'$, then the spaces $(V,\tau_{V,\kappa})$ and $(V,\tau_{V,\kappa'})$ are not isomorphic as topological vector spaces.
\end{prop}

\begin{proof}We have that $\{W < V: codim \ W < \kappa\} \subseteq \{W < V: codim \ W < \kappa'\}$, which implies that $\tau_{V,\kappa}$ is coarser than $\tau_{V,\kappa'}$.

Let $f: (V,\tau_{V,\kappa}) \rightarrow (V,\tau_{V,\kappa'})$ be an isomorphism of topological vector spaces. Let $W' < V$ such that $\kappa \leqslant codim \ W' < \kappa'$. Then $f^{-1}(W') \in \tau_{V,\kappa}$. So there exists $W \in \tau_{V,\kappa}$ such that $codim \ W < \kappa$ and $W < f^{-1}(W')$, which implies that $codim \ W' = codim \ f^{-1}(W') < \kappa$, contradicting the definition of $W'$. Thus $(V,\tau_{V,\kappa})$ and $(V,\tau_{V,\kappa'})$ are not isomorphic as topological vector spaces.
\end{proof}

From now on until \textbf{Theorem \ref{theoremone}}, we consider another property of topological vector spaces: let $V$ be a topological vector space with linear topology such that $Aut(V) = CAut(V)$.

\begin{lema}\label{equaldimandcodim} Let $W$ be an open subspace of $V$. If $W' < V$ such that $dim \ W' = dim \ W$ and $codim \ W' = codim \ W$, then $W'$ is open in $V$. 
\end{lema}

\begin{proof}Let $U, U' < V$ such that we can decompose $V$ as the inner direct sums $V = W \oplus U = W' \oplus U'$. We have that $codim \ W = dim \ U$ and $codim \ W' = dim \ U'$. Since $codim \ W = codim \ W'$, then $dim \ U = dim \ U'$. Let then $f: W' \rightarrow W$ and $g: U' \rightarrow U$ be isomorphisms. These maps induce an isomorphism $h: V \rightarrow V$ such that $h(W') = W$. The map $h$ is continuous by hypothesis and $W$ is open in $V$, which implies that $W' = h^{-1}(W)$ is also open in $V$.
\end{proof}

\begin{lema}\label{equaldim}Let $W$ be an open subspace of $V$ such that $dim \ W < dim \ V$. If $W' < V$ such that $dim \ W' = dim \ W$, then $W'$ is open in $V$. 
\end{lema}

\begin{proof}Suppose that $dim \ V < \aleph_{0}$. We have that $dim \ V = dim \ W + codim \ W =  dim \ W' + codim \ W'$ and $dim \ W' = dim \ W$, which implies that  $codim \ W' = codim \ W$. Since $W$ is open in $V$, then, by \textbf{Lemma \ref{equaldimandcodim}}, $W'$ is open in $V$.

Suppose that $dim \ V \geqslant \aleph_{0}$. We have that $dim \ V = dim \ W + codim \ W = \max\{dim \ W, codim \ W\}$ (because $dim \ V \geqslant \aleph_{0}$). Since $dim \ W < dim \ V$, then $codim \ W = dim \ V $. Analogously, $codim \ W' = dim \ V $. So $dim \ W' = dim \ W$ and $codim \ W' = codim \ W$, which implies, by \textbf{Lemma \ref{equaldimandcodim}}, that $W'$ is open in $V$.
\end{proof}

\begin{lema}\label{equalcodim}Let $W$ be an open subspace of $V$ such that $codim \ W < dim \ V$. If $W' < V$ such that $codim \ W' = codim \ W$, then $W'$ is open in $V$. 
\end{lema}

\begin{proof}Analogous to the proof of \textbf{Lemma \ref{equaldim}}.
\end{proof}

\begin{lema}\label{smallercodim}Let $W$ be an open subspace of $V$ such that $dim \ W < dim \ V$. If $W' < V$ such that $codim \ W' \leqslant codim \ W$, then $W'$ is open in $V$. 
\end{lema}

\begin{proof}Let $B$ be a basis of $V$ such that $B \cap W$ is a basis of $W$. Let $B' \subseteq B$ such that $B \cap W \subseteq B'$ and $\# B\setminus B' = codim \ W'$ (this is possible since $codim \ W' \leqslant codim \ W$). Take $W'' = \langle B' \rangle$. Since  $W < W''$ and $W$ is open in $V$, then $W''$ has non-empty interior, which implies, by Corollary 2.3 of \cite{Wa}, that $W''$ is open in $V$. We have that $codim \ W'' = codim \ W'$, which implies, by \textbf{Lemma \ref{equalcodim}}, that $W'$ is open in $V$.
\end{proof}

\begin{defi}Let $V$ be a topological vector space. We define $\zeta(V) = \sup\{\kappa: \exists W < V, W \text{ is open in $V$, } codim \ W = \kappa \}$.
\end{defi}

\begin{obs} We are considering just vector spaces with linear topology. So, the neighborhood filter in $V$ of the point $0$, $\mathcal{N}(0)$, is generated by the set $\{W < V: W \text{ is open in $V$ and } codim \ W \leqslant \zeta(V) \}$.
\end{obs}

The next propositions show what happens with the space $V$, for different values of $\zeta(V)$.

\begin{prop}\label{zetaisnotfinite} $\zeta(V) = 0$ or $\zeta(V) \geqslant \aleph_{0}$.
\end{prop}

\begin{proof}Suppose that $\zeta(V) = n \in \N$, with $n > 0$. Then there exists an open subspace $W$ such that $codim \ W = n$. Let $B$ be a basis of $V$ such that the set $B' = B \cap W$ is a basis of $W$. Note that $\# B \setminus B' = n$ since $codim \ W = n$. Take $b \in B \setminus B'$ and $b' \in B'$. Then we define $W' = \langle (B' \setminus \{b'\}) \cup \{b\} \rangle$. We have that $dim \ W' = dim \ W$ and $codim \ W' = codim \ W$, which implies, by \textbf{Lemma \ref{equaldimandcodim}}, that $W'$ is open in $V$. Then $W \cap W'$ is an open subspace of $V$ such that $codim \ W \cap W' > n$, since $W \cap W'$ is properly contained in $W$. This contradicts the fact that  $\zeta(V) = n$.

Thus $\zeta(V) = 0$ or $\zeta(V) \geqslant \aleph_{0}$.
\end{proof}

\begin{prop}\label{codimensionone}If $\zeta(V) = 0$, then $V$ has the trivial topology.  
\end{prop}

\begin{proof} Let $U \in \mathcal{N}(0)$. Then there exists an open subspace $W$ such that $W \subseteq U$. Since $\zeta(V) = 0$, $codim \ W = 0$, which implies that $W = U = V$. So $\mathcal{N}(0) = \{V\}$, which implies that $V$ has the trivial topology.
\end{proof}

\begin{prop}\label{supisnotmax}If there is no open subspace  of $V$ with codimension $\zeta(V)$, then $V$ has the $\zeta(V)$-codimension topology.
\end{prop}

\begin{proof}Let $\kappa < \zeta(V)$. Then there exists a cardinal $\kappa'$ such that $\kappa \leqslant \kappa' < \zeta(V)$ and there exists an open subspace $W < V$ with $codim \ W = \kappa'$. By \textbf{Lemma \ref{smallercodim}}, for every $W' < V$ with $codim \ W' = \kappa$, $W'$ is open in $V$. So we have that 
${\{W \! < \! V: W \text{ is open in $V$, } codim \ W \leqslant \zeta(V) \} \! = \! \{W \!< \!V: codim \ W < \zeta(V) \}}$. This implies that $V$ has the $\zeta(V)$-codimension topology.
\end{proof}

\begin{prop}\label{supismax}If there exists an open subspace of $V$ with codimension $\zeta(V)$, then $V$ has the $\zeta(V)^{+}$-codimension topology.
\end{prop}

\begin{proof}Suppose that $\zeta(V) < dim \ V$. Let $\kappa \leqslant \zeta(V)$. Then there exists a cardinal $\kappa'$ such that $\kappa \leqslant \kappa' \leqslant \zeta(V) < dim \ V$ and there exists an open subspace $W < V$ with $codim \ W = \kappa'$. By \textbf{Lemma \ref{smallercodim}}, for every $W' < V$ with $codim \ W' = \kappa$, $W'$ is open in $V$. Then we have that $\{W < V: W \text{ is open in $V$, } codim \ W \leqslant \zeta(V) \} = \{W < V: codim \ W < \zeta(V)^{+} \}$, which implies that $V$ has the $\zeta(V)^{+}$-codimension topology.

Suppose that $\zeta(V) = dim \ V$. Then there exists  $W < V$, open in $V$, such that $codim \ W = \zeta(V) = dim \ V$. Let $B$ be a basis of $W$ and $B'$ a set that is disjoint to $B$ and such that $B \cup B'$ is a basis of $V$. We have that $\# B' = codim \ W = dim \ V$ and $\# B \leqslant \# B'$. Take $B'' \subseteq B'$ such that $\# B'' = \# B$ and take $W'' = \langle  B'' \rangle$. We have that $dim \ W'' = dim \ W$ and $codim \ W'' = codim \ W$. So, by \textbf{Lemma \ref{equaldimandcodim}}, it follows that $W''$ is open in $V$, which implies that $0 = W \cap W''$ is also open in $V$. Then $V$ is discrete, which means that $V$ has the $\zeta(V)^{+}$-codimension topology.   
\end{proof}

Summarizing the above propositions, we get the following theorem:

\begin{teo}\label{theoremone}Let $V$ be a topological vector space with linear topology. Then $V$ has a codimension topology if and only if $Aut(V) = CAut(V)$.   
\end{teo}

\begin{proof}$(\Rightarrow)$ It is immediate from the fact that isomorphisms preserve the codimension of its subspaces.

$(\Leftarrow)$ We know, by \textbf{Proposition \ref{zetaisnotfinite}}, that $\zeta(V) = 0$ or $\zeta(V) \geqslant \aleph_{0}$. If $\zeta(V) = 0$, then $V$ has $1$-codimension topology, by \textbf{Proposition \ref{codimensionone}}.  So we can suppose that $\zeta(V) \geqslant \aleph_{0}$. If there is no open subspace of $V$ with codimension $\zeta(V)$, then, by \textbf{Proposition \ref{supisnotmax}}, $V$ has the $\zeta(V)$-codimension topology. If there exists an open subspace of $V$ with codimension $\zeta(V)$, then, by \textbf{Proposition \ref{supismax}}, $V$ has the $\zeta(V)^{+}$-codimension topology.
\end{proof}

There are some cases where every linear map is continuous, as the following propositions:

\begin{prop}\label{continuitytosmallspaces}Let $V$ and $V'$ be topological vector spaces with linear topology such that $V$ has the $\kappa$-codimension topology. If $dim \ V' < \kappa$, then
$CHom(V,V') = Hom(V,V')$.    
\end{prop}

\begin{proof}Let $f \in Hom(V,V')$. First we have that $V/f^{-1}(0)$ is isomorphic (as vector spaces) to a subspace of $V'$, which implies that $codim \ f^{-1}(0) = dim \ V/f^{-1}(0) \leqslant dim \ V' < \kappa$.

Let $W' < V'$. Then $f^{-1}(W) > f^{-1}(0)$, which implies that $codim \ f^{-1}(W) \leqslant codim \ f^{-1}(0) < \kappa$ and then $f^{-1}(W)$ is open in $V$. Since $V'$ has linear topology, the set of open subspaces generates the neighborhood filter of $0$ in $V'$, which implies that $f$ is a continuous map (since any preimage of a subspace is open).
\end{proof}

\begin{prop}\label{continuitybetweenboundedspaces}Let $V$ and $V'$ be topological vector spaces such that $V$ and $V'$ have the $\kappa$ and $\kappa'$-codimension topologies, respectively. If $\kappa' \leqslant \kappa$, then
$CHom(V,V') = Hom(V,V')$.    
\end{prop}

\begin{proof}Let $f \in Hom(V,V')$. Let $W'$ be an open subspace of $V'$. Then $codim \ W' < \kappa'$. By \textbf{Proposition \ref{preimagesandcodimension}}, $codim \ f^{-1}(W')  < \kappa' \leqslant \kappa$. So $f^{-1}(W')$ is open in $V$. Thus $f$ is continuous.
\end{proof}

\begin{obs}Note that \textbf{Theorem \ref{theoremone}} is an analogous version of Theorem 1 of \cite{So9} for topological vector spaces with linear topology. However, when we consider the whole set of homomorphisms instead of isomorphisms, the same analogy does not hold. For topological spaces we get that the trivial and the discrete spaces are the only ones where all maps to itself are continuous, Corollary 1 of \cite{So9}, but, as a special case of \textbf{Proposition \ref{continuitybetweenboundedspaces}}, we get that all vector spaces with a codimension topology have the property that all linear maps to itself are continuous.
\end{obs}

\section{Some properties of vector spaces with \\ codimension topology}

Here we present some properties of vector spaces with codimension topology.

\begin{prop}Let $V$ be a vector space with $\kappa$-codimension topology, with $\kappa > 1$, and $B$ a basis of $V$. Then $B$, with its subspace topology, is discrete.
\end{prop}

\begin{proof}Let $b \in B$. Take $W = \langle B \setminus \{b\}\rangle$. We have that $codim \ W = 1 < \kappa$, which implies that $W$ is an open subspace of $V$. But open subspaces are also closed (Corollary 2.3 of \cite{Wa}), which implies that $W$ is closed in $V$. Then $B \setminus \{b\} = W \cap B$ is a closed subset of $B$, which implies that $\{b\}$ is an open subset of $B$. Thus $B$ is discrete. 
\end{proof}

\begin{obs}Maybe it is surprising the fact that the subspace topology of $B$ is not the $co-\kappa$ topology, but $B$ has an uniform structure induced by the uniform structure of the topological vector space $V$ and co-cardinal spaces are not uniformizable in general.  In fact, the cofinite topology is $T_{1}$, is not Hausdorff (Counterexample 19 of \cite{SS}) and every $T_{1}$ uniform space is Hausdorff, which implies that a set with the cofinite topology cannot be uniformizable.
\end{obs}

\begin{prop}Let $V$ be a topological space with $\kappa$-codimension topology and $\kappa > 1$. If $W < V$, then $W$ is closed.
\end{prop}

\begin{proof}We have that $W = \bigcap \{W' < V: W < W', codim \ W' = 1\}$. Since $\kappa > 1$, all codimension $1$ subspaces are closed, which implies that $W$ is closed as well.
\end{proof}

As a consequence, we get:

\begin{cor}Let $V$ be a topological space with $\kappa$-codimension topology and $\kappa > 1$. Then $V$ is Hausdorff.
\end{cor}

\begin{proof}By the last proposition, $0$ is closed in $V$, which implies that $V$ is a Hausdorff space (Theorem 1.7 of \cite{Wa}).
\end{proof}

\begin{lema}\label{subspacetopology1}Let $V$ be a topological space with the $\kappa$-codimension topology and $W < V$. Then $W$ has the $\kappa'$-codimension topology, for some $\kappa' \geqslant \kappa$.
\end{lema}

\begin{proof}Let $f: W \rightarrow W$ be an isomorphism. Then we can extend $f$ to an isomorphism $f': V \rightarrow V$, which is continuous since $V$ has codimension topology. So $f$ is continuous, which implies that $W$ has the $\kappa'$-codimension topology for some cardinal $\kappa'$.

Let $Z < W$ with $codim_{W} Z < \kappa$ and $W' < V$ such that $V = W\oplus W'$ as vector spaces. Then $codim_{V} \ Z\oplus W' = codim_{W} Z < \kappa$, which implies that $Z\oplus W'$ is open in $V$ and then $Z =(Z\oplus W') \cap W$ is open in $W$. So $\kappa \leqslant \kappa'$.
\end{proof}

\begin{prop}Let $V$ and $V'$ be infinite dimensional vector spaces with $\kappa$-codimension topology, where $\kappa > 1$. Then $V \oplus V'$, with the coproduct topology, is isomorphic to $V \oplus V'$ with the $\kappa$-codimension topology, as topological vector spaces.
\end{prop}

\begin{obs}The coproduct $V\oplus V'$ and the product $V\times V'$ are isomorphic
as vector spaces and also they have the same topology (page 105 of \cite{BeH}).    
\end{obs}

\begin{proof}Consider $V \oplus V'$ with the coproduct topology. Let $f: V \oplus V' \rightarrow V \oplus V'$ be a linear map. Then there exists linear maps $f_{V}: V \rightarrow V\oplus V'$ and ${f_{V'}: V' \rightarrow V \oplus V'}$ that induce $f$. We can decompose $f_{V'} = f_{V',V}+f_{V',V'}$, where $f_{V',V}: V' \rightarrow V \oplus V'$ has its image contained in $V \oplus 0$ and the map $f_{V',V'}: V' \rightarrow V \oplus V'$ has its image contained in $0 \oplus V'$. By \textbf{Proposition \ref{continuitybetweenboundedspaces}} and the fact that $V\oplus 0$ is isomorphic to $V$ and $0 \oplus V'$ is isomorphic to $V'$ as topological vector spaces, then the maps $f_{V',V}$ and $f_{V',V'}$ are continuous, which implies that $f_{V'}$ is continuous. Analogously, we have that $f_{V}$ is continuous. So, by the universal property of the coproduct topology, we have that $f$ is continuous. So $V\oplus V'$ has the $\kappa'$-codimension topology, for some cardinal $\kappa'$. 

Since $V \oplus 0$ is a subspace of $V$ with the $\kappa$-codimension topology, then, by \textbf{Lemma \ref{subspacetopology1}}, $\kappa' \leqslant \kappa$. Note that $\aleph_{0} \leqslant \kappa'$, since $V\oplus V'$ is a Hausdorff space. Suppose that $\kappa' < \kappa$. Then take $Z < V$ and $Z' < V'$ such that $codim_{V} Z = codim_{V'} Z' = \kappa'$. Then $Z$ is open in $V$ and $Z'$ is open in $V'$, which implies that $Z\oplus Z'$ is open in $V \oplus V'$ (since the coproduct topology is equal to the product topology). We also have that $codim_{V\oplus V'}Z\oplus Z' = codim_{V}Z+codim_{V'}Z' = \kappa'+\kappa' = \kappa'$, which contradict the fact that $V\oplus V'$ has the $\kappa'$-codimension topology. So $\kappa' = \kappa$.
\end{proof}

\begin{cor}Let $V$ be an infinite dimension vector space endowed with the $\kappa$-codimension topology, where $\kappa > 1$. Then there exists a topological decomposition $V = Z \oplus Z'$, where $Z$ and $Z'$ are topologically isomorphic to $V$.
\end{cor}

\begin{proof}By the last proposition, $V \oplus V$ is isomorphic to $V$, as topological vector spaces. From this fact, the corollary is immediate.
\end{proof}

Now we are able to improve \textbf{Lemma \ref{subspacetopology1}}:

\begin{prop}\label{subspacetopology}Let $V$ be a topological space with the $\kappa$-codimension topology and $W < V$. Then $W$ also has the $\kappa$-codimension topology.
\end{prop}

\begin{proof}By \textbf{Lemma \ref{subspacetopology1}}, we have that W has the $\kappa'$-codimension topology, for some cardinal $\kappa' \geqslant \kappa$.

Suppose that $codim_{V} W < \kappa$. Then $W$ is open in $V$. Let $Z$ be an open subspace of $W$. Then $Z$ is open in $V$, which implies that $codim_{V} Z < \kappa$. Then $codim_{W}Z \leqslant codim_{V}Z < \kappa$. So $\kappa' = \kappa$.

Suppose that $dim \ W < \kappa$. Since $W$ has the $\kappa'$-codimension topology and $\kappa < \kappa'$, then $W$ is discrete, which implies that $W$ also has the $\kappa$-codimension topology.

Suppose that $dim \ W = \kappa$ and $codim_{V} W = \kappa$. Then there exists $W' < V$ such that $V = W \oplus W'$ as vector spaces and $dim \ W' = \kappa$. By the last corollary, there exists $Z < V$ and $Z'< V$ such that $V = Z \oplus Z'$ and $Z$ and $Z'$ are topologically isomorphic to $V$. We have that $dim \ V = dim \ W = dim \ Z = dim \ W' = dim \ Z'$, which implies that there exists an isomorphism $f: V \rightarrow V$, which is a homeomorphism, such that $f(W) = Z$ and $f(W') = Z'$. So $W$ is isomorphic to $Z$ as topological vector spaces and $Z$ is isomorphic to $V$ as topological vector spaces, which implies that $W$ is isomorphic to $V$ as topological vector spaces. Since $V$ has the $\kappa$-codimension topology, then $W$ has the $\kappa$-codimension topology.
\end{proof}

As a special case, we have:

\begin{cor}Let $V$ be a topological space with the $\kappa$-codimension topology and $W < V$. If $dim \ W < \kappa$, then $W$ is discrete.
\end{cor}

\begin{prop}\label{theyarenotlinearlycompact}Let $V$ be a topological vector space with $\kappa$-codimension topology and $dim \ V \geqslant \aleph_{0}$. Then $V$ is not linearly compact.
\end{prop}

\begin{proof}Every infinite dimensional linearly compact space has subspaces with codimension $1$ that are not open (by \textbf{Proposition \ref{linearcompacthasdense}}), which implies that $V$ is not linearly compact, for every $\kappa \neq 1$. If $\kappa = 1$, then $V$ is not linearly compact either, since linearly compact spaces are Hausdorff and $V$ is not. 
\end{proof}

\begin{prop}Let $V$ be a topological space with $\aleph_{0}$-codimension topology and $dim \ V \geqslant \aleph_{0}$. Then the Cauchy completion $\hat{V}$ is linearly compact. 
\end{prop}

\begin{proof}We have that $\hat{V} \cong \lim\limits_{\longleftarrow}\{V/W: W< V \text{ is open}\}$. But $\{V/W: W \text{ is open} \} = \{V/W: codim \ W < \aleph_{0}\}$, which implies that $\hat{V}$ is an inverse limit of finite dimensional vector spaces, which implies that it is linearly compact.
\end{proof}

\begin{prop}\label{bigfamily}Let $V$ be a vector space. The family $\mathcal{V} = \{V_{\kappa} =(V,\tau_{V,\kappa}): \aleph_{0} \leqslant \kappa \leqslant dim \ V^{+}\}$ satisfies the following properties:

\begin{enumerate}
    \item If $\aleph_{0} \leqslant \kappa, \kappa' \leqslant dim \ V^{+}$ and $\kappa \neq \kappa'$, then $V_{\kappa}$ and $V_{\kappa'}$ are not isomorphic as topological vector space.
    \item If $\aleph_{0} \leqslant \kappa, \kappa' \leqslant dim \ V^{+}$, then the profinite-dimensional completion of $V_{\kappa}$ and the profinite-dimensional completion of $V_{\kappa'}$ are isomorphic as topological vector spaces.
    \item If $\aleph_{0} \leqslant \kappa, \kappa' \leqslant dim \ V^{+}$, then the dual spaces $V_{\kappa}^{\circledast}$ and $V_{\kappa'}^{\circledast}$ are isomorphic as topological vector spaces.
\end{enumerate}
\end{prop}

\begin{obs}So we can construct arbitrarily big families satisfying these three properties.
\end{obs}

\begin{proof}The spaces of $\mathcal{V}$ are pairwise non isomorphic as topological vector spaces, by \textbf{Proposition \ref{differentspaces}}. 

We have that the profinite-dimensional completion of $V_{\kappa}$ is given by the inverse limit $\lim\limits_{\longleftarrow}\{V/W: W < V, codim \ W < \aleph_{0} \}$, since all subspaces with finite codimension are closed in $V_{\kappa}$. Note that $V/W$ is discrete, since all subspaces with finite codimension are also open in $V$. So, this inverse limit does not depend on the choice of $\kappa$, which implies that the profinite-dimensional completions of all $V_{\kappa}$ are isomorphic. 

As vector spaces, $V_{\kappa}^{\circledast} = CHom(V_{\kappa},k)$. By \textbf{Proposition \ref{continuitytosmallspaces}}, we have that $CHom(V_{\kappa},k) = Hom(V_{\kappa},k)$, which implies that the dual, as a vector space, does not depend on the choice of $\kappa$. The topology of $V_{\kappa}^{\circledast}$ is generated by the set $\{W^{\perp}: W < V, W \text{ is linearly compact}\}$. By \textbf{Proposition \ref{subspacetopology}}, all subspaces of $V$ also have the $\kappa$-codimension topology and, by \textbf{Proposition \ref{theyarenotlinearlycompact}} these subspaces cannot be linearly compact, unless they have finite dimension, which implies that $\{W^{\perp}: W < V, W \text{ is linearly compact}\} = \{W^{\perp}: W < V, dim \ W < \aleph_{0}\}$. Thus the topology of $V_{\kappa}^{\circledast}$ also does not depend of $\kappa$.
\end{proof}

As a consequence, we get:

\begin{cor}Let $V$ be a vector space with the $\kappa$-codimension topology, where $\kappa \neq 1$. Then $V^{\circledast}$ is linearly compact.
\end{cor}

\begin{proof}Let $V'$ be the same vector space of $V$ but with the $\kappa'$-codimension topology, where $\kappa' = dim \ V^{+}$. By the last proposition we have that $V^{\circledast}$ is isomorphic, as topological vector spaces, to $V'^{\circledast}$. But $V'$ is discrete, which implies that $V'^{\circledast}$ is linearly compact (Proposition 24.4 of \cite{BeH}). Thus $V^{\circledast}$ is linearly compact.
\end{proof}

\section{Some remarks about algebras}

In Functional Analysis there is the following famous result:

\begin{prop}(Johnson, Theorem 5.1.5 of \cite{Dal}) Let $A$ and $B$ be Banach algebras such that $B$ is semisimple. If $f: A \rightarrow B$ is an epimorphism of algebras, then $f$ is continuous. 
\end{prop}

However, the result above does not hold for the class of topological algebras with linear topology:

\begin{ex}\label{example1}Let $A$ be a non-discrete algebra with linear topology and $A_{dis}$ the same algebra with the discrete topology. Then the identity map ${id: A \rightarrow A_{dis}}$ is an epimorphism of algebras that is not continuous.
\end{ex}

\begin{defi}Let $k$ be a field, treated as a discrete topological ring. A pseudocompact $k$-algebra is a complete Hausdorff topological algebra possessing a fundamental system of neighborhoods of 0 consisting of two sided ideals with finite codimension that intersect in 0.
\end{defi}

If we consider only pseudocompact algebras, then Johnson's Theorem also does not hold:

\begin{ex}\label{example2}If $A$ is a pseudocompact algebra and $I$ is a maximal ideal that is dense on $A$ and $codim \ I = 1$ (see page 5 of \cite{IM} for an example), then  $A/I$ is isomorphic to $k$ as vector spaces. Let $\pi: A \rightarrow k$ be an algebra homomorphism such that $\ker \ \pi = I$ and where $k$ is discrete. We have that $\pi$ is an epimorphism and $k$ is simple, but this map is not continuous since $0$ is open in $k$ but $I = \pi^{-1}(0)$ is not open in $A$.
\end{ex}

\end{document}